\documentclass[11pt]{article}
\usepackage{amsmath,amsthm}
\usepackage{url}
\usepackage{times}
\usepackage{graphicx}
\usepackage{color}
\usepackage{amssymb}

\usepackage{amsmath,empheq}
\usepackage{hyperref}

\newtheorem{theorem}{Theorem}
\newtheorem{remark}[theorem]{Remark}
\newtheorem{corollary}[theorem]{Corollary}
\newtheorem{lemma}[theorem]{Lemma}
\newtheorem{definition}[theorem]{Definition}
\newtheorem{proposition}[theorem]{Proposition}

\newcommand{\comment}[1]{}

\usepackage{amssymb}
\usepackage{amsfonts}
\usepackage{times}

\makeatletter
\renewcommand*\env@matrix[1][*\c@MaxMatrixCols c]{%
	\hskip -\arraycolsep
	\let\@ifnextchar\new@ifnextchar
	\array{#1}}
\makeatother

\definecolor{gold}{rgb}{0.83, 0.69, 0.22}

\definecolor{greenncs}{rgb}{0.0,0.62,0.42}

\usepackage{makeidx}
\makeindex

\begin{document}
	
	\title{Compact Embedding Theorem Associated with Classical Weight Functions in Two Variables}
\author{M. K. NANGHO, B. J. NKWAMOUO and J.L. WOUKENG \\
			\\
			Department of Mathematics and Computer Science,\\ University of Dschang, PO Box 67, Dschang, Cameroon } 
	\maketitle
	
	
		\begin{abstract}
		For a classical weight function $\rho$ defined on a simply connected open subset $\Omega$ of $\mathbb{R}^2$ (either bounded or unbounded) with piecewise $C^1$ boundary, we prove density of the space of bivariate polynomials and compact embedding of the matrix-weighted Sobolev space $W^{1,2}(\Omega,\rho,\rho\Phi)$ in the weighted Lebesgue space   $L^2(\Omega,\, \rho)$. As an application, we investigate via a variational method, eigenvalue problem for a degenerate Helmholtz operator on  triangle.
		\vspace{5pt}
		
		\vspace{5pt}
	\end{abstract}
	
\maketitle
\pagestyle{myheadings}
\markboth{Maurice KENFACK NANGHO, Bleriod JIEJIP NKWAMOUO, and Jean Louis WOUKENG }{Compact embedding theorem associated with classical weight functions in two variables}

	\section{Introduction}
		Compact embeddings of Sobolev spaces are essential in
	the analysis of partial differential equations (cf. \cite{EDM}). For degenerate partial differential equations, this issue is fundamental and has attracted considerable attention for decades (cf. \cite{JH,BO,CRW, RM}). 
	Motivated by potential applications of the space of polynomials, $\mathcal{P}$, to the theory of weighted Sobolev spaces and second-order degenerate partial differential equations, we follow the definitions {\cite[Sections 1.62, 3.2(b)]{RA1}} of usual weak derivative and Sobolev spaces to define weighted Sobolev spaces involving weight functions, $\rho$, solutions of the first-order partial differential equation $div(\rho \Phi)=\rho\psi$ under Neumann-type condition. Moreover, we prove that each of  these spaces is compactly embedded into the weighted Lebesgue space $L^2$   associated with $\rho$, and provide an application. \\
	Let $\Omega$ be an open subset of $\mathbb{R}^2$ and $1\leq p< +\infty$. $L^p(\Omega)$ denotes the class of measurable functions $u$ defined on $\Omega$ such that $\int_{\Omega}|u|^pdx_1dx_2<+\infty,$ equipped with the norm 
	\begin{equation*}
		\|u\|_p^p=\int_{\Omega}|u|^pdx_1dx_2.
	\end{equation*}
	$W^{1,p}(\Omega)$ {denotes the set of all functions $u\in L^p(\Omega)$} so that $\partial_{x_1} u,\,\partial_{x_2} u\in L^p(\Omega)$, where $\partial_{x_1} u=\tfrac{\partial u}{\partial {x_1}}$ and $\partial_{x_1} u=\tfrac{\partial u}{\partial {x_2}}$ denote the weak (distributional) derivatives of $u$. That is, there exist $v_1,v_2\in L^p(\Omega)$ such that 
		\begin{equation*}
			\int_{\Omega}u\partial_{x_1} fdx_1dx_2=-\int_{\Omega} v_1fdx_1dx_2\;{\rm and}\;\int_{\Omega}u\partial_{x_2} fdx_1dx_2=-\int_{\Omega} v_2fdx_1dx_2,\,\, f\in C^{\infty}_c(\Omega),
		\end{equation*}
		which can also be written as 
		\begin{equation}\label{eqq1}
			\int_{\Omega}udiv(I_2f)dx_1dx_2=-\int_{\Omega} v^tI_2fdx_1dx_2,\,v=(v_1,v_2)^t,\,\, f\in C^{\infty}_c(\Omega),
	\end{equation}
	where $C^{\infty}_c(\Omega)$ denotes the set of infinitely differentiable functions on $\Omega$ with compact support in $\Omega$ and $I_2$ is the $(2,\,2)$-identity matrix. It is natural to refer to $v$, appearing in (\ref{eqq1}), as the weak gradient of $u$.
	Equipped with the norm 
	\begin{equation}
		\|u\|_{1,p}=\left(\int_{\Omega}|u|^pdx_1dx_2+\int_{\Omega} |\partial_{x_1} u|^p dx_1dx_2+\int_{\Omega} |\partial_{x_2} u|^p dx_1dx_2\right)^{1\over p}.
	\end{equation}
	$W^{1,p}\left(\Omega\right)$ is called a Sobolev space. It is well-known that the identity map from $W^{1,p}\left(\Omega\right)$ into $L^p\left(\Omega\right)$ is continuous. In other words, $W^{1,p}\left(\Omega\right)$ is embedded in the space $L^p\left(\Omega\right)$. When $\Omega$ is bounded with $C^1$ piecewise boundary, the embedding $W^{1,p}\left(\Omega\right)\rightarrow L^p(\Omega)$ is compact (cf. \cite[§5.7]{HB2011}). That is, every bounded sequence of $W^{1,p}\left(\Omega\right)$ has a subsequence converging to an element of $L^p\left(\Omega\right)$ with respect to the norm $\|.\|_{p}$ . This result is a fundamental tool for solving partial differential equations (PDEs) and variational problems involving limiting processes (cf. \cite{Evans, McOwen} and references therein). {For instance, consider the eigenvalue problem \cite{McOwen}
    \begin{equation}
		\left\{\begin{array}{ccc}
				-div(A\nabla u)+bu=\lambda\,u \quad \text{in} \; \Omega \label{seq1}\\
				\left(A\nabla u\right)\cdot\overrightarrow{n}=0 \quad \;\text{on} \; \partial \Omega,
			\end{array}\right.
		\end{equation}}
        where the matrix $A=\left(a_{ij}\right)_{i,j=1}^2$ is symmetric, the elliptic condition 

		\begin{equation}\label{ellc}
			\sum_{i,j=1}^{2}a_{ij}\zeta_i\zeta_j\geq \alpha |\zeta|^2,\alpha>0,
		\end{equation}
		 and the compactness of the embedding $W^{1,2}(\Omega)\rightarrow L^2(\Omega)$ {as well as} the boundedness of  function $b$ are key ingredients to prove the existence of a sequence of eigenvalues $0=\lambda_0<\lambda_1\leq \lambda_2\leq \lambda_3\leq \dots$ and associated eigenfunctions $\left\{u_n\right\}_{n\geq 0}$ that are weak solutions of the Neumann eigenvalue problem.
		
	When there {exists} $x_0\in\overline{\Omega}$ such that $\displaystyle{\lim_{x \to x_0}\sum_{i,j=1}^{2}a_{ij}(x)\zeta_i\zeta_j}=0$ for a non-zero $\zeta\in\mathbb{R}^n$, the elliptic condition (\ref{ellc}) is violated. In this case, the operator $Lu=-div(A\nabla u)+bu$ is said to be degenerate, and can no longer be analysed using classical Sobolev spaces, but rather requires weighted Sobolev spaces. This remains  true when $b$ is unbounded. These spaces are fundamental tools for analysing linear degenerate elliptic equations as well as PDEs in unbounded domains or with rough coefficients (cf. \cite{AK,BP, GU, MR} and {references therein}).  \cite{AK} is one of the first monographs dedicated to weighted Sobolev spaces. This book focuses on weighted inequalities, embeddings, and regularity theorems involving power or distance {weights}.  In the early 2000s by means of weighted Lebesgue space $L^2\left(\Omega,\, \rho\right)$,  the authors of \cite{BP} established a functional analytic setting suitable for handling elliptic equations with Neumann boundary conditions on  a closed  unbounded subset $K$ of $\mathbb{R}^d$ with $C^2$ boundary. In 2015 using matrix- weighted Sobolev spaces approach, Dariod D. Monticelli {et al.} \cite{MR} studied existence and spectral properties for weak solutions of Neumann and Dirichlet problems for linear degenerated elliptic operators with rough coefficients on bounded domain in $\mathbb{R}^2$. In these references as well as in \cite{SW, RM, DSKS} the authors considered matrix-weighted Sobolev space as completion of Lipschitz functions with respect to certain weighted Lebesgue norms. The approach to weighted Sobolev spaces developed in the present work is analogous to that of $W^{1,p}(\Omega)$. Before moving forward, let us recall and give some definitions:\\
	
	$L^2(\Omega,\, \rho)$ denotes the class of measurable functions $u:\Omega\rightarrow\mathbb{R}$ defined on $\Omega$ such that $$\int_{\Omega}|u|^2\rho dx_1dx_2<+\infty,$$ equipped with the norm 
	\begin{equation*}
		\|u\|_{2,\rho}^2=\langle u,u\rangle_{\rho},\;\langle u,v\rangle_{\rho}=\int_{\Omega}uv\rho dx_1dx_2.
	\end{equation*}
	Let $A$ be a symmetric positive-definite matrix on $\Omega$ and let $\mathcal{L}^2(\Omega,\,   \rho A)$ be the matrix-weighted Lebesgue space
	$$\mathcal{L}^2(\Omega,\,   \rho A)= \left\{ u= \begin{pmatrix}u_1\\ u_2\end{pmatrix}: \Omega \to \mathbb{R}^2 \; \text{measurable and}\, \displaystyle\int_\Omega u^t\rho A u \,dx_1dx_2< +\infty\right\}.$$ 
	The $\mathbb{R}$-valued functional defined on $\mathcal{L}^2(\Omega,\,   \rho A)$ by
	\begin{equation}\label{ie4}
		\|u\|_{2,\rho A}^2=\langle u,\,u\rangle_{\rho A},\;{\langle u,\,v \rangle_{\rho A}} =\int_{\Omega}u^t\rho A vdx_1dx_2,\;u,v\in \mathcal{L}^2(\Omega,\,  \rho A)
	\end{equation} 
	is a semi norm.
	Since $\|u-v\|_{2,\rho A}=0\Leftrightarrow\,u=v$ outside a set of measure zero,  $\|.\|_{2,\rho A}$ is a norm on $L^2(\Omega,\,  \rho A)$, the set of equivalence classes for the equivalence relation,  $u\mathcal{R}v\Leftrightarrow\,\|u-v\|_{2,\rho A}=0$, $u,v\in\mathcal{L}^2(\Omega,\,   \rho A)$. 
	We will denote by $\mathcal{P}$ the space of polynomials in two variables. $\mathcal{P}_n,\,n\geq 0,$ denotes subspace of $\mathcal{P}$ consisting of polynomials of degree at most $n$. $\mathcal{M}_2(\mathcal{P}_2)$ denotes the space of $(2,\,2)$-matrices with coefficients in $\mathcal{P}_2$.
	
	{ A weight function on $\Omega$ is a map $\rho:\Omega\rightarrow\mathbb{R}$  positive almost everywhere (a.e) and measurable such that  $$\int_{{\Omega}}\rho(x_1,x_2)dx_1dx_2<\infty.$$
		
			A  weight function $\rho$ on  $\Omega$ is a moment weight function if 
			\begin{equation}
				\quad \int_\Omega x_1^m x_2^n \rho(x_1,x_2)\,dx_1dx_2 < \infty,\; m,n\geq 0.
			\end{equation}
	}
	\begin{definition}\cite{NJN} \label{CWF}	Let  $\rho$ be a weight function on  $\Omega$.  $\Omega_j = \Omega \cap B(O,j)$, $j = 1,2, \dots$, and $\overrightarrow{n_j}$ an outward unit normal vector of $\partial \Omega_j$. $\rho$  is called classical if there is a symmetric matrix $\varPhi\in\mathcal{M}_{2}\left(\mathcal{P}_2\right)$, 	
		$\varPhi=\left(\begin{matrix}
			\phi_{1,1}&\phi_{1,2}\\
			\phi_{2,1}&\phi_{2,2}
		\end{matrix}\right)$, such that
		\begin{enumerate}
			\item {There exist two polynomials $\psi_i(x_1,x_2)=x^tD_i+E_i,\,i=1,2$, with $x=(x_1,x_2)^t,\,D_i\in\mathbb{R}\times\mathbb{R}$, $E_i\in \mathbb{R}$ and $det\left(D_1,\,D_2\right)\neq\,0$ such that $\rho$ satisfies the Pearson type equation
				\begin{equation}\label{te2a}
					div(\rho\varPhi)=\rho\left(\psi_{1},\,\psi_{2}\right).
			\end{equation} }
			\item Every element $u$ of $\mathcal{P}$ satisfies the Neumann boundary-type  condition   
			\begin{equation}\label{bc}
				\lim\limits_{j\rightarrow\infty}1_{\partial\Omega_j}\left(\rho \Phi\,\nabla\,u\right)\cdot\overrightarrow{n_j}=0.
			\end{equation}
			\item $\Phi$ satisfies the differential system 
		\begin{equation}\label{te1}
				\left\{\begin{array}{ccc}
\phi_{1,1}\partial_{x_1}\Phi+\phi_{2,1}\partial_{x_2}\Phi&=\Phi\nabla\left(\phi_{1,1},\phi_{2,1}\right),\\
    \phi_{1,2}\partial_{x_1}\Phi+\phi_{2,2}\partial_{x_2}\Phi&=\Phi\nabla\left(\phi_{1,2},\phi_{2,2}\right).
			\end{array}\right.
			\end{equation}
		\end{enumerate} 
	\end{definition}

	{In this work, we follow the definition of usual Sobolev spaces \cite[1.62, 3.2(b)]{RA1} to  introduce matrix-Sobolev spaces in two variables involving  classical weight functions $\rho$, defined on simply connected open subset $\Omega$ of $\mathbb{R}^2$, either bounded or unbounded, with piecewise $C^1$ boundary. We then prove that the closure of $\mathcal{P}$ in each of these matrix-Sobolev spaces is compactly embedded in $L^2(\Omega,\,\rho)$.
		More precisely for a classical weight function $\rho$ on $\Omega$, we establish the equality 
		\begin{equation*}
			\int_{\Omega}udiv(\rho\Phi v)dx_1dx_2=-\int_{\Omega}(\nabla u)^t\rho\Phi vdx_1dx_2,\;v=(v_1,\,v_2)^t\in \mathcal{P}\times \mathcal{P},
		\end{equation*}
		where  $u$ is a smooth function on the closure of $\Omega$ such that $u\in L^2(\Omega,\,\rho)$ and $\nabla u\in L^2(\Omega,\, \rho\Phi)$, and use it to generalise  (\ref{eqq1})  {as follows:}} for $u\in L^2(\Omega,\, \rho)$, if there is $h=(h_1,h_2)$ belonging to the matrix-Lebesgue space $L^2(\Omega,\, \rho\Phi)$ such that 
	\begin{equation}\label{dder}
		\int_{\Omega}udiv(\rho\Phi v)dx_1dx_2=-\int_{\Omega}h^t\rho\Phi vdx_1dx_2,\;v=(v_1,\,v_2)^t\in \mathcal{P}\times \mathcal{P},
	\end{equation}
	then $h$ is called the weak (distributional) gradient of $u$. Using this definition of weak gradient and following the definition of usual Sobolev spaces $W^{1,2}(\Omega)$ (cf. \cite[Definition 3.2]{RA1}), we introduce the matrix-weighted Sobolev space
	\begin{align*}
		W^{1,2}(\Omega,\,  \rho, \rho \Phi)= \left\{ u \in L^2(\Omega,\,  \rho) \; \text{and} \; \int_{\Omega}h^t\rho\Phi\nabla udx_1dx_2<\infty\right\},
	\end{align*}
	{equipped with the norm $\lVert u \rVert_{W^{1,2}(\Omega,\, \rho,\,\rho\Phi)} =\left( \int_\Omega u^2 \rho \,dx_1dx_2 + \int_\Omega (\nabla u)^t \rho \Phi (\nabla u) \,dx_1dx_2 \right)^{\frac{1}{2}}$ 
	where $\nabla$ is understood in the weak (distributional) sense (\ref{dder}) and $\Omega$ {is a} simply connected open subset of $\mathbb{R}^2$ with piecewise $C^1$ boundary. Moreover, denoting by 
	\begin{equation*}
	W(\Omega,\,\rho,\rho\Phi), {\rm the\, closure\, of}\, \mathcal{P} \,{\rm in \,the \,space} \,W^{1,2}(\Omega,\,  \rho, \rho \Phi)
		,\end{equation*}} we prove that:
	\begin{enumerate}
		\item the space of bivariate polynomials, $\mathcal{P}$ is dense in $L^2(\Omega,\, \rho)$,
		\item the embedding of $W^{1,2}\left(\Omega,\, \rho,\,\rho\Phi\right)$  in  $L^2(\Omega,\, \rho)$ is compact,
		\item 	$L^2(\Omega,\,\rho)$ has an orthogonal basis formed by   eigenfunctions of the generalised Helmholtz operator
		\begin{eqnarray*}
			Lu&=&-x_1(1-x_1)\partial_{x_1}^2u+2x_1x_2\partial_{x_1x_2}^2u-x_2(1-x_2)\partial_{x_2}^2u-(\alpha+1-(\alpha+\beta+\gamma+3)x_1)\partial_{x_1}u\\
			&&-(\beta+1-(\alpha+\beta+\gamma+3)x_2)\partial_{x_2}u+(2+x_1^2+x_2^2)u.
		\end{eqnarray*}  $\rho$ is the two-variable  weight function $$\rho(x_1,x_2)=x_1^{\alpha} {x_2}^{\beta}\left(1-x_1-x_2\right)^{\gamma},\,\alpha,\beta,\gamma>-1$$ and $\Omega$ is the triangle $\Omega=\left\{(x_1,x_2);x_1,\,x_2\geq 0\,{\rm and}\, x_1+x_2\leq 1\right\}$. Clearly, $L$ is degenerate at $(0,\,0),\,(0,\,1)$ and $(1,\,0)$. Those  eigenfunctions of $L$  are elements of $W(\Omega,\,\rho,\,\rho\Phi)$, their eigenvalues are positive and can be ordered in a sequence tending to $\infty$.
	\end{enumerate} 
	In the sequel, we will consider $\Omega$ to be a simply connected open subset of $\mathbb{R}^2$ with piecewise $C^1$ boundary and $\rho$ to be a weight function on $\Omega$ satisfying the matrix Pearson equation (\ref{te2a}) and the Neumann boundary-type condition (\ref{bc}). Under these conditions, we assume without loss of generality that $\rho$ is a moment weight function.
	
	\section{Matrix-weighted Sobolev space $W^{1,2}(\Omega,\, \rho,\rho \Phi)$}

	\begin{lemma} The space $L^2(\Omega,\,\rho \Phi)$ endowed with the norm (\ref{ie4}) is a Hilbert space.
	\end{lemma}
	
	\begin{proof}
		Since $\Phi$ is a positive-definite matrix-valued polynomial on $\Omega$, its eigenvalues are positive polynomials on $\Omega$. Since $\rho$ is a moment {weight} function, we obtain
		\begin{equation*}
			\int_{\Omega}\left|\rho\,\Phi\right|_{op}dx_1dx_2<+\infty\,,
		\end{equation*}
		where $\lvert . \rvert_{op}$ denotes the norm on $(2,\,2)$-matrices,
		$|\Phi|_{op}=sup\left\{\left |\lambda\right|; \lambda\, {\rm eigenvalue\; of}\, \Phi\right\}$. Thus,  it follows, from  \cite[Theorem 4]{SW} with $\mathcal{Q}=\rho \Phi$, that  $L^2(\Omega,\,  \rho \Phi)$ is a Hilbert space.  
	\end{proof}
	
	{\begin{lemma}\label{inlem} \hspace{1cm}\\
			\begin{enumerate}
				\item Let  $u$ and $v$ be elements of $L^2(\Omega,\,\rho\Phi)$. Then, $\|u^t\Phi v\|_{L^1(\Omega,\,\rho)}\leq \|u\|_{L^2(\Omega,\,\rho\Phi)}\|v\|_{L^2(\Omega,\,\rho\Phi)}.$ 
				\item Let $u$ be an element of $L^2(\Omega,\,\rho)$. Then, for all $v\in\mathcal{P}\times\mathcal{P}$, there exists a constant $C\geq 0$ such that  $$\int_{\Omega}|u div(\rho\Phi )v|dx_1dx_2\leq C \| u\|_{L^2(\Omega,\,\rho)}.$$
			\end{enumerate}
		\end{lemma}
		\begin{proof}
			Since the matrix $\Phi$ is symmetric and positive-definite, the function $H(u,v) =u^t\Phi v$, $u,v\in L^2(\Omega,\,\rho\Phi)$ is a scalar product. Therefore, $\vert u^t\Phi v\rvert \leq \sqrt{u^t\Phi u}\sqrt{v^t\Phi v}$. Integrating both sides with respect to $\rho dx_1dx_2$ and using the H\"older inequality, we obtain the first item.
			As for the second item, we first write by direct computation 
			$div(\rho\Phi v)=div(\rho\Phi)v+\left(\Phi_{11}\partial_{x_1}v_1+\Phi_{21}\partial_{x_2}v_1+\Phi_{12}\partial_{x_1}v_2+\Phi_{22}\partial_{x_1}v_2\right)\rho$ and take into account the Pearson equation (\ref{te2a}) to obtain
			\begin{equation}\label{div2}
				div(\rho\Phi v)=K\rho,
			\end{equation} where $K=\psi_{1}v_1+\psi_2v_2+\Phi_{11}\partial_{x_1}v_1+\Phi_{21}\partial_{x_2}v_1+\Phi_{12}\partial_{x_1}v_2+\Phi_{22}\partial_{x_2}v_2$ is a polynomial. Therefore,
			$$\int_{\Omega}|udiv(\rho\Phi v)| dx_1dx_2=\int_{\Omega}|uK|\rho dx_1dx_2.$$ We finally use the H\"older inequality to deduce $$\int_{\Omega}|udiv(\rho\Phi v)| dx_1dx_2\leq \left(\int_{\Omega}K^2\rho dx_1dx_2\right)^{\tfrac{1}{2}}\left(\int_{\Omega}u^2\rho dx_1dx_2\right)^{\tfrac{1}{2}}.$$ Since $\rho$ is a moment weight, $C=\left(\int_{\Omega}K^2\rho dx_1dx_2\right)^{\tfrac{1}{2}}$ is finite.
		\end{proof}
	}
	\begin{remark}\label{bcremark}\hspace{1 cm}\\
		\begin{enumerate}
			\item If $\Omega$ is bounded, (\ref{bc}) becomes $1_{\partial\Omega}\left(\rho\Phi\nabla v\right)\cdot\overrightarrow{n}=0$, $v\in\mathcal{P}$.
			\item (\ref{bc}) can also be rewritten as 	$\lim\limits_{j\rightarrow\infty}1_{\partial\Omega_j}\left(\rho \Phi\,v\right)\cdot\overrightarrow{n_j}=0,\, v\in\mathcal{P}\times\mathcal{P}$.
		\end{enumerate}
	\end{remark}

	\begin{proposition}\label{dderiv}
		Let $u$ be a smooth function on $\overline{\Omega}$ such that $u\in L^2(\Omega,\,  \rho)$ and $\nabla u\in L^2(\Omega,\,  \rho\Phi)$. Then
		\begin{equation}\label{dderiv0}
			\int_{\Omega}u div(\rho\Phi v)dx_1dx_2=-\int_{\Omega}(\nabla u)^t\rho\Phi vdx_1dx_2,\,v=(v_1,v_2)^t\in \mathcal{P}\times\mathcal{P}.
		\end{equation}
	\end{proposition}
	\begin{proof}
		Observing that: \begin{equation}\label{deq}
			div(u\rho\Phi\,v)=udiv(\rho\Phi\,v)+\left(\nabla u\right)^t\rho\Phi\,v
		\end{equation} and integrating both sides on $\Omega$ we obtain 
		\begin{equation}\label{ideq}
			\int_{\Omega}udiv(\rho\Phi\,v)dx_1dx_2=\int_{\Omega}div(u\rho\Phi\,v)dx_1dx_2-\int_{\Omega}\left(\nabla u\right)^t\Phi\,v\rho dx_1dx_2.
		\end{equation}
		\begin{itemize}
		\item 	If $\Omega$ is bounded, from the divergence theorem, $\int_{\Omega}div(u\rho\Phi\,v)dx_1dx_2= {\int_{\partial\Omega}\left(\rho\Phi\nabla v\right)\cdot\overrightarrow{n}\,dx_1dx_2 }$. Using the boundary condition for bounded domain (cf. the first item of the Remark\, \ref{bcremark})  and the fact that  $u$ is defined on $\overline{\Omega}$, we have $1_{\partial\Omega}u\left(\rho\Phi v\right)\cdot\overrightarrow{n}=0$. Therefore, $\int_{\Omega}div(u\rho\Phi\,v)dx_1dx_2=0$.
		
		\item {	If $\Omega$ is unbounded, then \\$\displaystyle{\lim_{j \to \infty }1_{\Omega_j}div(u\rho\Phi\,v)}=1_{\Omega}div(u\rho\Phi\,v)$ and $\lvert1_{\Omega_j}div(u\rho\Phi\,v)\rvert\leq \lvert 1_{\Omega}div(u\rho\Phi\,v)\rvert$ because $\Omega_j$ is an {increasing} sequence of subsets of $\Omega$ such that $\bigcup\Omega_j=\Omega$.\\
		Combining the relation (\ref{deq}) and the Lemma \ref{inlem}, we derive the existence of a constant $C\geq 0$ such that
			$$\int_{\Omega}\lvert div(u\rho\Phi\,v)\rvert dx_1dx_2 \leq C \| u\|_{L^2(\Omega,\,\rho)}+\| \nabla u\|_{L^2(\Omega,\,\rho\Phi)}\| v\|_{L^2(\Omega,\,\rho\Phi)}.$$ Therefore, $\int_{\Omega}\lvert div(u\rho\Phi\,v)\rvert dx_1dx_2<+\infty$ for $u\in L^2(\Omega,\,\rho)$ and $\nabla u\in L^2(\Omega,\,\rho\Phi)$. So, by the dominated convergence theorem on the sequence $\left\{1_{\Omega_j}div(u\rho\Phi\,v)\right\}_{j\geq 1}$ it follows that:} 
		
		$$\int_{\Omega}div(u\rho\Phi\,v)dx_1dx_2=\lim_{j \to \infty }\int_{\Omega_j}div(u\rho\Phi\,v)dx_1dx_2.$$ Since  $\Omega_j,\,j=1,2,.. .$ is bounded, we deduce from the divergence theorem $$\int_{\Omega_j}div(u\rho\Phi\,v)dx_1dx_2=\int_{\partial\Omega_j}\left(u\rho\Phi\,v\right)\cdot\overrightarrow{n_j}dx_1dx_2.$$ Therefore, 
		
		$$\int_{\Omega}div(u\rho\Phi\,v)dx_1dx_2=\lim_{j \to \infty }\int_{{\Omega}} 1_{\partial\Omega_j}\left(u\rho\Phi\,v\right)\cdot\overrightarrow{n_j}dx_1dx_2.$$
		
		 From the boundary condition for unbounded domain (cf. the second item of the Remark\,\ref{bcremark}) $$\displaystyle{\lim_{j \to \infty}1_{\partial\Omega_j}\left(u\rho\Phi v\right)\cdot\overrightarrow{n_j}}=0$$ whenever $v\in \mathcal{P}\times \mathcal{P}$. So, to obtain (\ref{dderiv0}) it suffices to prove that  
		\begin{equation}\label{e10}
			\lim_{j \to \infty }\int_{{\Omega}} 1_{\partial\Omega_j}\left(u\rho\Phi\,v\right)\cdot\overrightarrow{n_j}dx_1dx_2=\int_{{\Omega}}  \lim_{j \to \infty } 1_{\partial\Omega_j}\left(u\rho\Phi\,v\right)\cdot\overrightarrow{n_j}dx_1dx_2.
		\end{equation} 
		Let  $\{f_j\}_{j\geq 1}$ {be the sequence of functions defined by} 
		$f_j=1_{\partial\Omega_j}\left(u\rho\Phi\,v\right)\cdot\overrightarrow{n_j},j=1,2,...\, .$ According to the Cauchy-Schwarz inequality, we deduce that:  $$|f_j|\leq 1_{\partial\Omega_j} |u|\|\overrightarrow{n_j}\|\|\Phi\,v\|\rho\leq 1_{\overline{\Omega}}|u|\|\Phi\,v\|\rho$$ for $\|\overrightarrow{n_j}\|=1$. Therefore, the use  of the {H\"older} inequality yields   
		$$\int_{\Omega}|u|\|\Phi\,v\|\rho dx_1dx_2\leq \left(\int_{\Omega}u^2\rho dx_1dx_2\right)^{1\over 2}\left(\int_{\Omega}\|\Phi\,v\|^2\rho dx_1dx_2\right)^{1\over 2}.$$ Since $u\in L^2(\Omega,\, \rho)$ and $\|\Phi\,v\|^2$ is a polynomial, the right-hand side is finite. So $\left\{f_j\right\}_{j\geq 1}$ is a sequence of integrable functions on $\Omega$. Moreover, there exists a positive integrable function, $h=|u|\|\rho\Phi\,v\|$, on $\Omega$ such that $|f_j|\leq h$. Therefore, from the dominated convergence theorem, (\ref{e10}) is satisfied. Thus $\int_{\Omega}div(u\rho\Phi\,v)dx_1dx_2=0.$ Using (\ref{ideq}), we obtain $$\int_{\Omega}udiv\left(\rho\Phi\right)dx_1dx_2=-\int_{\Omega}\left(\nabla u\right)^t\rho\Phi vdx_1dx_2\; v\in \mathcal{P}\times \mathcal{P}.$$
	\end{itemize}
	\end{proof}
	This Proposition allows us to define weak gradient as well as matrix-weighted Sobolev space $W^{1,2}(\Omega,\, \rho,\rho\Phi)$ {as follows:}
	
	\begin{definition}\label{dweakd} Let  $u$ be an element of $L^2(\Omega,\, \rho)$. The weak gradient of $u$ when it exists is a function $h\in L^{2}(\Omega,\,\rho\Phi)$ such that 
		\begin{equation}\label{weakd}
\int_{\Omega}u div(\rho\Phi v)dx_1dx_2=-\int_{\Omega}h^t\rho\Phi vdx_1dx_2,\,v=(v_1,v_2)^t\in \mathcal{P}\times\mathcal{P}.
		\end{equation}
	\end{definition}
	
	\begin{definition}\label{WSS}\hspace{1cm}\\
		\begin{enumerate}
			\item $W^{1,2}(\Omega,\,\rho,\,\rho\Phi)$ is the matrix-weighted Sobolev space
			$$W^{1,2}(\Omega,\,\rho,\,\rho\Phi)=\left\{u \in L^2(\Omega,\,  \rho);\,\nabla u \in L^{2}\left(\Omega,\,\rho\Phi\right)\right\}$$
			endowed with the norm \begin{equation}\label{norm} \lVert u \rVert_{W^{1,2}(\Omega,\, \rho,\,\rho\Phi)} =\left( \int_\Omega u^2 \rho \,dx_1dx_2 + \int_\Omega (\nabla u)^t \rho \Phi (\nabla u) \,dx_1dx_2 \right)^{\frac{1}{2}},\end{equation}
			where $\nabla u$ is understood in the weak (distributional) gradient of Definition \ref{dweakd}.
			\item $W(\Omega,\,\rho,\,\rho\Phi)$ is the closure of the space of polynomials in two variables, $\mathcal{P}$, in\\ $W^{1,2}(\Omega,\,\rho,\,\rho\Phi)$ with respect to the norm (\ref{norm}).
		\end{enumerate}
	\end{definition}
	
	The norm (\ref{norm}) is associated with the inner product
	$$\langle u, v \rangle_{W^{1,2}(\Omega,\, \rho,\,\rho\Phi)} = \int_\Omega u v \rho \,dx_1dx_2 + \int_\Omega (\nabla u)^t \rho \Phi (\nabla v) \,dx_1dx_2,\;\; u,v\in W^{1,2}(\Omega,\, \rho,\,\rho\Phi).$$
	
	In the sequel, for an element $u\in W(\Omega,\, \rho,\rho\Phi)$, we will denote its norm of $W^{1,2}(\Omega,\, \rho,\rho\Phi)$ by\\ $\lVert . \rVert_{W(\Omega,\, \rho,\,\rho\Phi)}$.
	
	\begin{proposition} $W^{1,2}(\Omega,\,\rho,\,\rho\Phi)$ is a Hilbert space.
	\end{proposition}
	\begin{proof}
		Let $\{u_n\}_{n \geq 1}$ be a Cauchy sequence in $ W^{1,2}(\Omega,\,  \rho, \rho \Phi)$.  $\{u_n\}_{n \geq 1}$  and $\left\{\nabla u_n  \right\}_{n \geq 1}$ are Cauchy sequences of $L^2(\Omega,\,  \rho)$ and $L^2(\Omega,\,  \rho \Phi)$ respectively. Since these two spaces are Hilbert spaces,  $\{u_n\}_{n \geq 1}$ converges in $L^2(\Omega,\,  \rho)$ to some limit that we denote by $u$, and $\left\{ \nabla u_n  \right\}_{n \geq 1}$ converges in $L^2(\Omega,\,  \rho \Phi)$ to some limit that we denote by $w= (w_1, w_2)^t$. Since $u_n\in W^{1,2}(\Omega,\,\rho\,\rho\Phi)$, for all $v=(v_1,\,v_2)$, $v_i\in \mathcal{P},i=1,2,$
		\begin{equation*}
			\int_{\Omega}u_ndiv(\rho\Phi\,v)dx_1dx_2=-\int_{\Omega}\left(\nabla u_n\right)^t\rho\Phi v dx_1dx_2.
		\end{equation*}
		From the equation (\ref{div2}) $div(\rho\Phi\,v)=K\rho$, where $K$ is a polynomial. Therefore, 
		\begin{equation}\label{heq}
			\int_{\Omega}u_nK\rho dx_1dx_2=-\int_{\Omega}\left(\nabla u_n\right)^t\rho\Phi v dx_1dx_2.
		\end{equation}
		Since the sequences $\{u_n\}_{n\geq 1}$ and $\left\{ \nabla u_n  \right\}_{n \geq 1}$ converge to $u$ and $w$  in $L^2(\Omega,\,\rho)$ and $L^2(\Omega,\,  \rho \Phi)$ respectively, they weakly converge to $u$ and $w$ in $L^2(\Omega,\,\rho)$ and $L^2(\Omega,\,  \rho \Phi)$ respectively. So,
		\begin{eqnarray*}
			\lim_{n \to \infty}\int_{{\Omega}}u_nK\rho dx&=&\int_{{\Omega}}uK\rho dx,\;{\rm for}\;K\in\mathcal{P}\subset L^2(\Omega,\rho),\\
				\lim_{n \to \infty}\int_{{\Omega}}(\nabla u_n)^t\Phi v\rho dx&=&\int_{{\Omega}}w^t\Phi\,v\rho dx,\;{\rm for}\;v\in\mathcal{P}\times \mathcal{P}\subset L^2(\Omega,\Phi\rho).
		\end{eqnarray*}
		 Therefore,
		letting $n\to \infty$ in  (\ref{heq}), it follows that: 	$\int_{\Omega}u K\rho dx_1dx_2=-\int_{\Omega}w^t\rho\Phi v dx_1dx_2.$ Hence	$$\int_{\Omega}udiv(\rho\Phi\,v)  dx_1dx_2=-\int_{\Omega}w^t\rho\Phi v dx_1dx_2.$$ By Definition \ref{dweakd}, $w$ is a weak gradient of $u$. Thus, $\{u_n\}_{n \geq 1}$ converges to $u$ in $W^{1,2}\left(\Omega,\,\rho,\,\Phi\rho\right)$.
	\end{proof}

	\begin{lemma}\label{Lemma1}Let $C^1_c(\Omega)$ be the space of smooth functions on $\Omega$ with compact support. Then $C^1_c(\Omega)$ is a subset of $W(\Omega,\,\rho,\,\rho\Phi)$.	
	\end{lemma}
	\begin{proof}
		Let $u$ be an element of $C_c^1(\Omega)$. There exists a compact $K\subset\Omega$ such that the support of $u$ is included in $K$. From Whitney \cite{WH}, there is a sequence {$\{p_n\}_{n \geq 1}$}  of polynomials in two variables such that {$\{p_n\}_{n \geq 1}$}, {$\{\frac{\partial p_n}{\partial{x_1}}\}_{n \geq 1}$}  and {$\{\frac{\partial p_n}{\partial{x_2}}\}_{n \geq 1}$}  converge uniformly on $K$ to $u$, $\frac{\partial u}{\partial{x_1}}$ and $\frac{\partial u}{\partial{x_2}}$ respectively. The following inequality can be obtained by direct computation 
		\begin{align*}
			&\lefteqn{\int_{\Omega}(\nabla (u-p_n))^t\rho\Phi\nabla (u-p_n)dx_1dx_2}\\
			&=\int_K(\nabla (u-p_n))^t\rho\Phi\nabla (u-p_n)dx_1dx_2\\
			&\leq M  \left(\sup_{(x_1,x_2)\in K}|\partial_{x_1}u(x_1,x_2)-\partial_{x_1}p_n(x_1,x_2)|+\sup_{(x_1,x_2)\in K}|\partial_{x_2}u(x_1,x_2)-\partial_{x_2}p_n(x_1,x_2)|\right)^2,
		\end{align*}
		where $M=\displaystyle{\max_{1\leq i,j\leq 2}\int_K|\phi_{i,j}|\rho dx_1dx_2}$.  Therefore, $$\displaystyle{\lim_{n\to \infty}\int_{\Omega}(\nabla (u-p_n))^t\rho\Phi\nabla (u-p_n)dx_1dx_2=0},$$ for $\{\frac{\partial p_n}{\partial{x_1}}\}_{n \geq 1}$  and $\{\frac{\partial p_n}{\partial{x_2}}\}_{n \geq 1}$  converge uniformly on $K$ to $\frac{\partial u}{\partial{x_1}}$ and $\frac{\partial u}{\partial{x_2}}$ respectively. Moreover,    $$\lim_{n \to \infty}\int_{\Omega}|u(x_1,x_2)-p_n(x_1,x_2)|^2\leq \lim_{n \to \infty}\sup_{(x_1,x_2)\in K}|u(x_1,x_2)-p_n(x_1,x_2)|^2\int_K\rho dx_1dx_2=0.$$ 
		Thus $\displaystyle{\lim_{n\to \infty} \|f-p_n\|_{W(\Omega,\, \rho,\,\rho\Phi)}}=0$. That is, $u\in W(\Omega,\, \rho,\rho\Phi)$. 
	\end{proof}
	\begin{theorem} \hspace{ 1 cm}\\
		\begin{enumerate}
			\item $W(\Omega,\,\rho,\,\rho\Phi)$ is dense and embedded into $L^2\left(\Omega,\,\rho\right)$.
			\item  The space of polynomials in two variables, $\mathcal{P}$, is dense in $L^2(\Omega,\,\rho)$.
		\end{enumerate}
		
	\end{theorem}
	\begin{proof}
		From the Lemma \ref{Lemma1}, $C_c^1(\Omega)\subset W(\Omega,\, \rho,\rho\Phi)\subset L^2(\Omega,\, \rho)$. Since $C_c^1(\Omega)$ is dense in $L^2(\Omega,\,  \rho)$, the density of $W(\Omega,\, \rho,\rho\Phi)$ in $L^2(\Omega,\, \rho)$ follows. The embedding is obtained using the inequality $\|u\|_{L^2(\Omega,\,\rho)}\leq \|u\|_{W(\Omega,\, \,\rho,\,\rho\Phi)}$ for all $u\in W(\Omega,\,\rho,\,\rho\Phi)$. The second statement follows from the first and the definition of $W(\Omega,\,  \rho,\rho\Phi)$.
	\end{proof} 
	Therefore, by the definition \cite[p.143]{HB2011}, it follows that
	\begin{corollary} Any family of polynomials that is orthonormal with respect to $\rho$ forms a Hilbert basis of $L^2(\Omega,\,\rho)$.
	\end{corollary}
	
	\section{Compact embedding theorem}

\begin{theorem}\label{th1}
	The space	$W^{1,2}(\Omega,\,  \rho, \rho \Phi)$ is compactly embedded in $L^2(\Omega,\,  \rho)$.
\end{theorem}
\begin{proof} We will split the proof into two steps: first, we assume that $\Omega$ is bounded and second, we assume that it is unbounded.
	\begin{enumerate}
		\item[Step 1.] $\Omega$ is a bounded subset of $\mathbb{R}^2$.\\
		Let $\{u_n\}_{n \geq 1}$ be a bounded sequence in $W^{1,2}(\Omega,\,  \rho, \rho \Phi)$. We prove the existence of a subsequence of $\{u_n\}_{n \geq 1}$ that converges in  $L^2(\Omega,\,  \rho)$.\\		
		Since $ W^{1,2}(\Omega,\,  \rho, \rho \Phi)$ is continuously embedded in $L^2(\Omega,\,  \rho)$, the sequence $\{u_n\}_{n \geq 1}$ is also bounded in $L^2(\Omega,\,  \rho)$. Then, there exists a constant  $C > 0$ such that for all $n \geq 1$, $\lVert u_n \rVert_{L^2(\Omega,\,  \rho)} < C$, that is ,
		\begin{equation}\label{3eq1}
			\lVert \rho^\frac{1}{2} u_n \rVert_{L^2(\Omega)} < C.
		\end{equation}
		Given  {$n \geq 1$}, we set 
		$ \tilde{v}_n = \rho^\frac{1}{2} \tilde{u}_n, \qquad \text{where} \qquad 
		\tilde{u}_n= \begin{cases}
			u_n(x) & \text{if $x \in \Omega$}\\
			0&  \text{if $ x\in \mathbb{R}^2 \setminus \Omega$}.
		\end{cases}.$\\
		$\tilde{u}_n \in L^2(\mathbb{R}^2)$ for all $n\geq 1$. Let $\eta$ be the function defined by :
		$$\eta(x) = \begin{cases}
			ke^{\frac{1}{ \lVert x\rVert^2 -1}} & \text{if $ \lVert x \rVert < 1 $}\\
			0 & \text{if $ \lVert x \rVert \geq 1$} 
		\end{cases},$$ where $k= \left(\int_{\mathbb{R}^2} \eta(x)\,dx \right)^{-1}$ and  {$\eta$} is an infinitely  differentiable function with compact support. Let {$(\eta_\varepsilon)_{\varepsilon>0}$ be the sequence of mollifiers functions associated with  {$\eta$}.  $\eta_\varepsilon =  \varepsilon^{-2} \eta\left( \frac{x}{\varepsilon} \right)$, $\varepsilon>0$}.   The sequence of smooth functions on $\mathbb{R}^2$, {$\{\tilde{v}_{n}^{\varepsilon}\}_{\varepsilon>0},$} where $\tilde{v}_{n}^{\varepsilon}(x) = (\eta_\varepsilon* \tilde{v}_n)(x),$ converges to $\tilde{v}_n$  in $L^2(\mathbb{R}^2)$ (cf.\cite[Theorem 2.29]{RA1}). Moreover, $\tilde{v}_{n}^{\varepsilon}$ is differentiable on $\Omega$ and its differential at $x\in\Omega$ is 
		
		$$D\tilde{v}_{n}^{\varepsilon}(x)(h)=h_1\frac{\partial \tilde{v}_{n}^{\varepsilon}}{\partial x_1} (x)+h_2\frac{\partial \tilde{v}_{n}^{\varepsilon}}{\partial x_2} (x),\;h\in\mathbb{R}^2.$$  \\
		Differentiating the function $\tilde{v}_{n}^{\varepsilon}$  with respect to the variable $x_1$ (respectively $x_2$), we obtain
		\begin{eqnarray*}
			\frac{\partial \tilde{v}_{n}^{\varepsilon}}{\partial x_1} (x)= \varepsilon^{-2} \int_\Omega \frac{\partial}{ \partial x_1}\left[ \eta\left( \frac{x-y}{\varepsilon} \right)\right] \tilde{v}_n(y) \,dy & = \varepsilon^{-3} \int_\Omega \left({\partial \eta_{\varepsilon}\over\partial x_1}\right)\left( \frac{x-y}{\varepsilon} \right) \tilde{v}_n(y) \,dy,\label{3eq2} \\
			\frac{\partial \tilde{v}_{n}^{\varepsilon}}{\partial x_2} (x)= \varepsilon^{-2} \int_\Omega \frac{\partial}{ \partial x_2} \left[\eta\left( \frac{x-y}{\varepsilon} \right)\right] \tilde{v}_n(y) \,dy & = \varepsilon^{-3} \int_\Omega \left({\partial \eta_{\varepsilon}\over\partial x_2}\right)\left( \frac{x-y}{\varepsilon} \right) \tilde{v}_n(y) \,dy.\label{3eq3}
		\end{eqnarray*} 
		Hence, using the {H\"older} inequality as well as  the relation (\ref{3eq1}), we derive the following
		\begin{equation}\label{3eq4} 
			\left| \frac{\partial\, \tilde{v}_{n}^{\varepsilon}}{\partial x_1} (x) \right| \leq \,C \varepsilon^{-3}\, |\Omega|^{\frac{1}{2}}\, \sup_{z \in \Omega} \left| \left({\partial \eta_{\varepsilon}\over\partial x_1}\right)(z) \right|\; \text{and}\;
			\left| \frac{\partial\, \tilde{v}_{n}^{\varepsilon}}{\partial x_2} (x) \right| \leq  \,C \varepsilon^{-3}\, |\Omega|^{\frac{1}{2}}\,\sup_{z \in \Omega} \left| \left({\partial \eta_{\varepsilon}\over\partial x_2}\right)(z) \right|.
		\end{equation}
		Therefore, \[\sup_{h\in\mathbb{R}^2,h\neq 0}\tfrac{\|D\tilde{v}_{n}^{\varepsilon}(x)(h)\|}{\|h\|}\leq K_1+K_2\] with $K_i=\,C \varepsilon^{-3}\, |\Omega|^{\frac{1}{2}}\,\sup_{z \in \Omega} \left| \left({\partial \eta_{\varepsilon}\over\partial x_i}\right)(z) \right|,\,i=1,2$. {Hence, by the Mean Value Theorem, we derive the inequality}
		
		\begin{equation}\label{3eq6}
			\lvert \tilde{v}_{n}^{\varepsilon}(x) - \tilde{v}_{n}^{\varepsilon}(y) \rvert \leq (K_1 +K_2) \lVert x-y \rVert ,\qquad  x,\,y \in \Omega.
		\end{equation}
		Furthermore, {since $\left| \tilde{v}_{n}^{\varepsilon}  \right| \leq\varepsilon^{-2}\sup_{z \in \mathbb{R}^2}|\eta(z)|\int_{\Omega}|\tilde{v}_{n}(y)|dy$, using the H\"older's inequality together with the relation (\ref{3eq1}) we deduce}
		$$\left| \tilde{v}_{n}^{\varepsilon}  \right| \leq \, C \,  \varepsilon^{-2}  |\Omega|^{\frac{1}{2}},$$
		for {$\sup_{z \in \mathbb{R}^2}|\eta(z)|\leq 1$}. Then we derive from the  Azerl\`a-Ascoli theorem  (cf. \cite[Theorem 1.33]{RA1}) that there exists a subsequence {$\{\tilde{v}_{n_k}^{\varepsilon}\}_{k \geq 1}$ of $\{\tilde{v}_{n}^{\varepsilon}\}_{n \geq 1}$} which converges in $\mathcal{C}(\overline{\Omega})$ to some limit denoted by $\tilde{v}^{\varepsilon}$.
		Since $\lvert \Omega\rvert< \infty$, 
		$ \int_{\Omega} \mid \tilde{v}_{n_k}^{\varepsilon}-\tilde{v}^{\varepsilon} \mid^{2} dx \leq  |\Omega|\, \sup_{x \in \Omega}\lvert\left( \tilde{v}_{n_k}^{\varepsilon}-\tilde{v}^{\varepsilon}\right)(x) \rvert^{2} $.
		Therefore, $ \tilde{v}_{n_k}^{\varepsilon} \longrightarrow \tilde{v}^{\varepsilon}$ in $L^2(\Omega)$. 
		
		Since $\{\tilde{v}_{n_k}^{\varepsilon}\}_{k \geq 1}$ converges in $L^2(\Omega)$, it is a Cauchy sequence. Hence, for a constant $\delta>0$, there exists {$k_0 \geq 1$} such that $$j,k\geq k_0\Rightarrow \lVert \tilde{v}_{n_k}^{\varepsilon}-\tilde{v}_{n_j}^{\varepsilon}\rVert_{L^2(\Omega)} \leq \delta.$$
		Thus, from the inequality 
		
			\begin{equation*}
			\lVert\tilde{v}_{n_k} - \tilde{v}_{n_j}\rVert_{L^2(\Omega)}\leq \lVert \tilde{v}_{n_k}-\tilde{v}_{n_k}^{\varepsilon}\rVert_{L^2(\mathbb{R}^2)} +  \lVert \tilde{v}_{n_k}^{\varepsilon}-\tilde{v}_{n_j}^{\varepsilon}\rVert_{L^2(\Omega)} + \lVert \tilde{v}_{n_j}^{\varepsilon}-\tilde{v}_{n_j}\rVert_{L^2(\mathbb{R}^2)},
		\end{equation*}
		we obtain
		\begin{equation}\label{3eq8}
			k,j\geq k_0 \implies \lVert\tilde{v}_{k} -\tilde{v}_{j} \rVert_{L^2(\Omega)} \leq  \lVert \tilde{v}_{n_k}-\tilde{v}_{n_k}^{\varepsilon}\rVert_{L^2(\mathbb{R}^2)} +\delta  +  \lVert \tilde{v}_{n_j}^{\varepsilon}-\tilde{v}_{n_j}\rVert_{L^2(\mathbb{R}^2)}.
		\end{equation}
		
		{Letting $\varepsilon\to 0$, it follows that} 
		$$k,j\geq k_0\Rightarrow \lVert \tilde{v}_{n_k}-\tilde{v}_{n_j}\rVert_{L^2(\Omega)} \leq \delta.$$ 
		{Therefore, $\left\{\tilde{v}_{n_k}\right\}_{k\geq 1}$  is a Cauchy sequence in $L^2(\Omega)$. So, there exists $\tilde{v} \in L^2(\Omega)$ such that $\tilde{v}_{n_k}     
			\longrightarrow \tilde{v}$ in $L^2(\Omega)$. That is
			\begin{align*}
				0=\lim\limits_{k\rightarrow \infty}  \int_{\Omega}^{} \mid \tilde{v}_{n_k}-\tilde{v} \mid^{2}dx_2&
				=\lim\limits_{k\rightarrow\infty}\int_{\Omega}^{}\mid \rho^{\frac{1}{2}} u_{n_k}-v\mid^2dx
				&=\lim\limits_{k\rightarrow\infty}\int_{\Omega}^{}\mid  u_{n_k}-v\rho^{\frac{-1}{2}}\mid^2\rho dx.
		\end{align*}}
		Hence {$u_{n_k}\longrightarrow v\rho^{\frac{-1}{2}}$} in $L^2(\Omega,\, \rho)$. Therefore, $\{u_{n_k}\}_{k \geq 1}$ is a subsequence of $\{u_n\}_{n \geq 1}$ which converges in $L^2(\Omega,\, \rho)$.  
		\item[Step 2.]  $\Omega$ is unbounded.\\ Let us define $B_m=B(O,m)$  and $D_m= \overline{\mathbb{R}^2\setminus B_m}$. Since $\Omega\cap B_m$ is bounded
		for all {$m \geq 1$}, the set of restrictions of elements of  $N=\left\{u_n,n\geq 1\right\}$ on $\Omega\cap B_m$ is a {precompact} subset of $L^2({\Omega\cap B_m, \rho})$.
		Let $\delta$ be a positive number. For $u\in N$, the sequence $\left\{g_m\right\}_{m\geq 1}$, $g_m=1_{\Omega\cap D_m}\lvert u\rvert^2$, converges pointwise to $1_{\emptyset} \lvert u\rvert^2=0$ and   $\lvert g_m\rvert \leq  1_{\Omega}\lvert u \rvert^2$. Moreover {$\int_{\Omega} 1_{\Omega}\lvert u \rvert^2\rho dx =\int_{\Omega}^{}\lvert u \rvert^2\rho dx < +\infty$},  for  $N\subset L^2(\Omega,\,  \rho)$. Then it follows from the Lebesgue dominated convergence theorem  that {$\lim\limits_{m\to +\infty}\int_{\Omega}g_m\rho\, dx= \int_{\Omega} 0 \rho\,dx=0$.}
		Thus, for all $\delta>0,$ there exists $m_0\in \mathbb{N}$ such that $m\geq m_ 0\Rightarrow \int_{\Omega\cap D_m} \lvert u\rvert^2\rho dx < \delta.$
		We deduce from \cite[Proposition 2.1]{JH} that $N$ is {precompact} in $L^2(\Omega,\,\rho).$
	\end{enumerate}
\end{proof}

\begin{remark}{We can observe that we have not made use of (\ref{te1}).} 
\end{remark} 

	\section{Application}
	A generalised Helmholtz operator also called a Helmholtz operator with variable coefficients (cf. \cite{CD})  is a second-order partial differential operator of the form $L=-G+\lambda$, where L is a second-order elliptic operator with variable coefficients.
	In this section, we investigate eigenvalue problem associated with the generalised  Helmholtz operator
		\begin{eqnarray}\label{Helm}
		Lu&=&-x_1(1-x_1)\partial_{x_1}^2u+2x_1x_2\partial_{x_1x_2}^2u-x_2(1-x_2)\partial_{x_2}^2u-(\alpha+1-(\alpha+\beta+\gamma+3)x_1)\partial_{x_1}u\\
		&&-(\beta+1-(\alpha+\beta+\gamma+3)x_2)\partial_{x_2}u+(2+x_1^2+x_2^2)u\nonumber
	\end{eqnarray}
	
on the triangular domain  $\Omega=\left\{(x_1,x_2);x_1,\,x_2\geq 0\,{\rm and}\, x_1+x_2\leq 1\right\}$. 
	
$$\lim_{x \to x_0}x_1(1-x_1)\zeta_1^2-2x_1x_2\zeta_1\zeta_2+x_2(1-x_2)\zeta_2^2
=0,\;x_0\in\left\{(0,\,0),\,(1,\,0),\,(0,\,1)\right\}\subset\partial\Omega.$$
So, the elliptic condition is violated at some points of $\overline{\Omega}$. Therefore, the operator L is degenerate and cannot be analysed using classical Sobolev spaces, but rather requires weighted Sobolev spaces.\\  Multiplying both sides of (\ref{Helm}) by the  two-variable  weight function 
	\begin{equation}\label{TW}
		\rho(x_1,x_2)=x_1^{\alpha}x_2^{\beta}\left(1-x_1-x_2\right)^{\gamma},\;\; (x_1,\,x_2)\in\Omega,\, \alpha,\,\beta,\,\gamma >-1.
	\end{equation} and using the identity \begin{equation}\label{peq}
		div(\rho(x_1,\,x_2)\Phi)=\left(\alpha+1-\left(\alpha+\beta+\gamma+3\right)x_1,\,\beta+1-\left(\alpha+\beta+\gamma+3\right)x_2\right)\rho(x_1,\,x_2),
	\end{equation} where $\Phi$ is the symmetric positive-definite matrix  
	\begin{equation}\label{matrix}
		\varPhi=\left(\begin{matrix}
			x_1(1-x_1)&-x_1x_2\\
			-x_1x_2&x_2(1-x_2)
		\end{matrix}\right),
	\end{equation}
	 (\ref{Helm}) becomes\\ $$L{u}=-\frac{1}{\rho(x_1,x_2)}div(\rho(x_1,x_2)\Phi\nabla u)+(2+{x_1^2+x_2^2})u.$$
	 
	Therefore, under the Neumann boundary condition in the first item of the Remark\,\ref{bcremark}, the second-order partial differential equation $Lu=f$ reads
	
	\begin{eqnarray}\label{system}
	\left\{\begin{array}{ccc}
			&-\frac{1}{\rho(x_1,x_2)}div(\rho(x_1,x_2)\Phi\nabla u)+(2+{x_1^2+x_2^2})u=f\label{Helm1},\,{\rm on}\,\Omega,\, \\
			\\
			&
			\left(\rho \Phi\,\nabla\,u\right)\cdot\overrightarrow{n}=0,\;{\rm on}\,\partial\Omega  \label{Hembc},
		\end{array}\right.
	\end{eqnarray}
	where $\varPhi$ is the matrix (\ref{matrix}) and $\rho$ is the weight function  (\ref{TW}) .\\
	The objective of this section is to use a variational method to prove the existence and uniqueness of a weak solution of (\ref{system})  in the space $W\left(\Omega,\, \rho,\,\rho\Phi\right)$ and derive by means of  \cite[Theorem II.6.5]{BOYER} the existence of an orthogonal basis of $L^2(\Omega,\, \rho)$ formed by solutions of the eigenvalue problem associated with (\ref{system}). 

	\subsection{Weak formulation}
	This subsection is devoted to weak formulation of (\ref{system}).
	One can check by direct computation that the weight function (\ref{TW}) is classical (see \cite{NJN}). Since the domain $\Omega$ is bounded, from the Remark\,\ref{bcremark}, $\rho$ satisfies 
	\begin{equation}\label{pbc}
		\left(\rho \Phi\,\nabla\,p\right)\cdot\overrightarrow{n}=0,\;{\rm on}\,\partial\Omega,\,p\in \mathcal{P}.
	\end{equation}
	
	\begin{lemma}\label{Lemmabc} Let $u$ be an element of $W(\Omega,\,\rho ,\,\rho \Phi).$  $u$ satisfies the boundary condition in (\ref{Hembc}) a.e on $\partial\Omega$.
	\end{lemma}
	\begin{proof}
		Let  $u$ be an element of $W(\Omega,\,\rho ,\,\rho \Phi).$  There exists a sequence of polynomials {$\{p_k\}_{k \geq 1}$} converging to $u$, with respect to the $W^{1,2}(\Omega,\,\rho ,\,\rho \Phi)$ norm. From (\ref{pbc}), 
		
		\begin{align*}
			1_{\partial\Omega}\left(\rho \Phi\,\nabla\,u\right)\cdot\overrightarrow{n}=1_{\partial\Omega}\left(\rho \Phi\,\nabla\,(u-p_k)\right)\cdot\overrightarrow{n},\,k\geq 1.
		\end{align*}
		
		Using the Cauchy-Schwarz inequality, we deduce
		\begin{align*}
			\lvert1_{\partial\Omega}\left(\rho \Phi\,\nabla\,u\right)\cdot\overrightarrow{n}\rvert\leq1_{\partial\Omega}\|\left(\rho \Phi\,\nabla\,(u-p_k)\right)\| \|\overrightarrow{n_j}\|,\;k\geq 1.
		\end{align*}
		
		Since $\|\overrightarrow{n}\|=1$,	$\lvert1_{\partial\Omega}\left(\rho \Phi\,\nabla\,u\right)\cdot\overrightarrow{n}\rvert\leq 1_{\overline{\Omega}}\|\left( \Phi\,\nabla\,(u-p_k)\right)\|\rho ,\;k\geq 1.$	
		Integrating both sides over $\overline{\Omega}$ and using  {H\"older} Inequality as well as the fact that 
			$$\int_{\Omega}\rho(x_1,x_2)dx_1dx_2=\frac{\Gamma(\alpha+1)\Gamma(\beta+1)\Gamma(\gamma+1)}{\Gamma(\alpha+\beta+\gamma+3)}$$ (see \cite[p.80]{Yuan}) we deduce
		\begin{equation}\label{e11}
			\int_{\overline{\Omega}}\left |1_{\partial\Omega}\left(\rho  \Phi\,\nabla\,u\right)\cdot\overrightarrow{n}\right |dx_1dx_2\leq\sqrt{\frac{\Gamma(\alpha+1)\Gamma(\beta+1)\Gamma(\gamma+1)}{\Gamma(\alpha+\beta+\gamma+3)}} \sqrt{\int_{\Omega}\|\Phi \nabla (u-p_k)\|^2\rho dx_1dx_2}.
		\end{equation}
		
		One can observe that.
		\begin{equation*}
			\|\Phi \nabla (u-p_k)\|^2=\left(\nabla \left(u-p_k\right)\right)^t\Phi^t\Phi\nabla (u-p_k)\leq \lvert \Phi\rvert_{op}\left(\nabla\left(u-p_k\right)\right)^t\Phi\nabla (u-p_k), 
		\end{equation*} 
		where $\lvert . \rvert_{op}$ denotes the norm on $(2,\,2)$-matrices,
		$|\Phi|_{op}=sup\left\{\left |\lambda\right|; \lambda\, {\rm eigenvalue\, of}\, \Phi\right\}$.
		Since the coefficients of the matrix $\Phi$ are polynomials and the domain $\Omega$ is bounded, there exists
		$C>0$ such that $\lvert \Phi\rvert_{op}\leq C$. Therefore,
		\begin{equation*}
			\|\Phi \nabla (u-p_k)\|^2\leq C\left(\nabla (u-p_k)\right)^t\Phi \nabla \left(u-p_k)\right).
		\end{equation*}
		Integrating both sides on $\Omega$, with respect to the measure $\rho dx_1dx_2$, and using the definition of the $W\left(\Omega,\,\rho,\,\rho\Phi\right)$ norm,  we deduce
		$$\int_{\Omega}\|\Phi \nabla (u-p_k)\|^2\rho dx_1dx_2\leq C\|u-p_k\|_{W\left(\Omega,\,\rho,\,\rho\Phi\right)}^2.$$ 
		Combining with (\ref{e11}) we get
			\begin{equation*}
				\int_{\overline{\Omega}}\left |1_{\partial\Omega}\left(\rho  \Phi\,\nabla\,u\right)\cdot\overrightarrow{n}\right |dx_1dx_2
				\leq \sqrt{\frac{C\Gamma(\alpha+1)\Gamma(\beta+1)\Gamma(\gamma+1)}{\Gamma(\alpha+\beta+\gamma+3)}} \|u-p_k\|_{W\left(\Omega,\,\rho ,\,\rho \Phi\right)},
		\end{equation*}
		where $\Gamma$ denotes the usual gamma function. Letting $n\to \infty$ and using the fact that {$\{p_k\}_{n \geq 1}$} converges to $u$ with respect to the  $W\left(\Omega,\,\rho ,\,\rho \Phi\right)$ norm, it follows that $\int_{\overline{\Omega}}\left |1_{\partial\Omega}\left(\rho  \Phi\,\nabla\,u\right)\cdot\overrightarrow{n}\right |dx_1dx_2=0$. 
		Therefore,
		$$\left |1_{\partial\Omega}\left(\rho  \Phi\,\nabla\,u\right)\cdot\overrightarrow{n}\right |=0\; {\rm a.e}.$$ 	  
	\end{proof}
	\begin{proposition} If the function $f$ in (\ref{system}) belongs to $L^2(\Omega,\,\rho)$ and there exists an element 
			$u$ of \\ $\mathcal{C}^2(\overline{\Omega})\cap W(\Omega,\,\rho,\,\rho\Phi)$ solution of the system (\ref{system}), then $u$ satisfies the weak formulation
			\begin{equation}\label{weakHelm}
				\int_{\Omega}(\nabla v)^t\rho \Phi\nabla udx_1dx_2+\int_{\Omega}\left(2+{x_1^2+x_2^2}\right)vu\rho dx_1dx_2=\int_{\Omega}vf\rho dx_1dx_2,\;v\in W(\Omega,\,\rho,\rho\Phi).
			\end{equation}
		\end{proposition}
		\begin{proof}
			If $u\in \mathcal{C}^2(\overline{\Omega})\cap W(\Omega,\,\rho,\,\rho\Phi)$ is a solution of the system (\ref{system}), then multiplying the first equation of (\ref{system}) by $v\rho $,  $v\in\mathcal{P}$ and using the relation
			$$div(v\rho \Phi\nabla u)=vdiv(\rho \Phi\nabla u)+(\nabla v)^t\rho \Phi\nabla u,$$ we obtain 
			\begin{equation}\label{Helm2}
				-div(v\rho \Phi\nabla u)+(\nabla v)^t\rho \Phi\nabla u+\left(2+{x_1^2+x_2^2}\right)vu\rho = vf\rho .
			\end{equation}
			Integrating both sides over $\Omega$, it follows that		
			\begin{align*}
				\lefteqn{-\int_{\Omega}div(v\rho \Phi\nabla u)dx_1dx_2}&\\
				&+\int_{\Omega}(\nabla v)^t\rho \Phi\nabla udx_1dx_2+\int_{\Omega}\left(2+{x_1^2+x_2^2}\right)vu\rho dx_1dx_2=\int_{\Omega}vf\rho dx_1dx_2.
			\end{align*}
			Using the divergence theorem as well as  the Lemma\,\ref{Lemmabc}, we deduce $$\int_{\Omega}div(v\rho \Phi\nabla u)dx_1dx_2=\int_{{\Omega}}\left(v\rho\Phi\nabla u\right)\cdot\overrightarrow{n}=0.$$ Therefore, $u$ satisfies (\ref{weakHelm}) for all $v\in\mathcal{P}$.\\ 
			If $v\in W(\Omega,\,\rho ,\rho\Phi)$, then there exists a sequence $\{p_k\}_{k\geq 1}$ of $\mathcal{P}$ converging to $v$ with respect to the $W(\Omega,\,\rho ,\rho\Phi)$ norm. 
			Since 
			\begin{equation*}
				\left\lvert\int_{\Omega}(\nabla p_k)^t\rho \Phi\nabla u dx_1dx_2-\int_{\Omega}(\nabla v)^t\rho \Phi\nabla u dx_1dx_2 \right\rvert=\left\lvert\int_{\Omega}(\nabla (p_k-v))^t\rho \Phi\nabla u dx_1dx_2\right\rvert,
			\end{equation*}
			the use of Cauchy-Schwarz inequality on the scalar product of $L^2(\Omega,\,\rho \Phi)$ yields 
			\begin{align*}
				\left	\lvert\int_{\Omega}(\nabla p_k)^t\rho \Phi\nabla u-\int_{\Omega}(\nabla v)^t\rho \Phi\nabla u \right\rvert&\leq\|p_k-v\|_{L^{2}\left(\Omega,\,\rho \Phi\right)}\|u\|_{L^{2}\left(\Omega, \,\rho \Phi\right)}\\
				&\leq \|p_k-v\|_{W\left(\Omega,\, \rho ,\,\rho \Phi\right)}\|u\|_{W\left(\Omega,\, \rho ,\,\rho \Phi\right)}.
			\end{align*}
			Letting $k\to\infty$ and using the fact that $\{p_k\}_{k\geq 1}$  converges in $W(\Omega,\,\rho ,\rho\Phi)$ with respect to\\ the $W(\Omega,\,\rho ,\rho\Phi)$ norm, it follows that 
			\begin{equation*}
				\lim_{k \to \infty }\int_{\Omega}(\nabla p_k)^t\rho \Phi\nabla u dx_1dx_2=\int_{\Omega}(\nabla v)^t\rho \Phi\nabla u dx_1dx_2.
			\end{equation*}
			 
			 Observing that 
			 \begin{equation}\label{bcc}
			 	0\leq x_1^2+x_2^2\leq 1,\; (x_1,\,x_2)\in\Omega,
			 \end{equation}
			 we have
			\begin{eqnarray*}
			\left\lvert	\int_{\Omega}\left(2+{x_1^2+x_2^2}\right)p_ku\rho   dx_1dx_2-\int_{\Omega}\left(2+{x_1^2+x_2^2}\right)vu\rho  dx_1dx_2\right\rvert&=&\left\lvert\int_{\Omega}\left(2+{x_1^2+x_2^2}\right)\left( p_k-v\right)u\rho dx_1dx_2\right\rvert\\
			&\leq&\,3 \left\lvert\int_{\Omega}\lvert p_k-v\rvert \lvert u\rvert\rho dx_1dx_2\right\rvert\\
			&\leq&3\|p_k-v\|_{W\left(\Omega,\, \rho ,\,\rho \Phi\right)}\|u\|_{W\left(\Omega,\, \rho ,\,\rho \Phi\right)}.
			\end{eqnarray*}
			Letting $k\to\infty$, we deduce
			\begin{equation*}
				\lim_{k \to \infty }\int_{\Omega}\left(2+{x_1^2+x_2^2}\right)p_ku\rho   dx_1dx_2\rho   dx_1dx_2=\int_{\Omega}\left(2+{x_1^2+x_2^2}\right)vu\rho   dx_1dx_2  dx_1dx_2.
			\end{equation*}
			In a similar way, we prove that:
			\begin{equation*}
				\lim_{k \to \infty }\int_{\Omega}p_kf\rho   dx_1dx_2=\int_{\Omega}vf\rho\rho  dx_1dx_2.
			\end{equation*}
			
			Therefore, since
			\begin{equation*}
				\int_{\Omega}(\nabla p_k)^t\rho \Phi\nabla udx_1dx_2+\int_{\Omega}\left(2+{x_1^2+x_2^2}\right)p_ku\rho dx_1dx_2=\int_{\Omega}p_kf\rho dx_1dx_2,\,k\geq 1,
			\end{equation*}
			letting $k\to\infty$, we conclude that:
			\begin{equation*}
				\int_{\Omega}(\nabla v)^t\rho \Phi\nabla udx_1dx_2+\int_{\Omega}\left(2+{x_1^2+x_2^2}\right)vu\rho dx_1dx_2=\int_{\Omega}vf\rho dx_1dx_2.
			\end{equation*}
	\end{proof}
	
	\begin{remark}
		Any $u\in W\left(\Omega,\,\rho ,\,\rho \Phi\right)$ satisfying (\ref{weakHelm}) is called weak solution of (\ref{system}).
	\end{remark}
	\subsection{Existence of weak solution and orthogonal basis}
	\begin{proposition}\label{EXIST} For $f\in L^2\left(\Omega,\,\rho \right)$, there exists  in  $W(\Omega,\,\rho ,\,\rho \Phi)$ a unique weak solution, $u$, of  (\ref{system}).
	\end{proposition}
	\begin{proof}
		For the proof of this proposition, we use \cite[Theorem II.2.5]{BOYER}.
		Let $a$ be the operator defined   from $W(\Omega,\,\rho ,\,\rho \Phi)\times W(\Omega,\,\rho ,\,\rho \Phi)$ to $\mathbb{R}$ by 
		$$ a(u,v)=	\int_{\Omega}(\nabla v)^t\rho \Phi\nabla udx_1dx_2+\int_{\Omega}\left(2+{x_1^2+x_2^2}\right)uv\rho dx_1dx_2.$$  The functional $a$  is clearly bilinear. Since $2+x_1^2+x_2^2>1, (x_1,\,x_2)\in\Omega$,
		$$ a(u,u)>\int_{\Omega}(\nabla u)^t\rho \Phi\nabla udx_1dx_2+\int_{\Omega}u^2\rho dx_1dx_2=\|u\|_{W(\Omega,\,\rho ,\,\rho \Phi)}^2,\;u\in W(\Omega,\,\rho ,\,\rho \Phi).$$
		 Therefore, 
		$a$ is coercive.\\ For all $u,v\in W\left(\Omega,\,\rho ,\,\rho \Phi\right).$ Since $0\leq x_1^2+x_2^2\leq 1$ on $\Omega$, it follows that
		$$|a(u,\,v)|\leq |\langle \nabla u,\,\nabla v\rangle_{L^2(\Omega,\,\rho \Phi)}|+3\langle u,\,v\rangle_{L^2(\Omega,\,\rho )}.$$ The use of Cauchy-Schwarz Inequality leads to
		\begin{align*}
			|a(u,\,v)|\leq \|\nabla u\|_{L^2(\Omega,\,\rho \Phi)}\|\nabla v\|_{L^2(\Omega,\,\rho \Phi)}+3\|u\|_{L^2(\Omega,\,\rho )}\|v\|_{L^2(\Omega,\,\rho )}.
		\end{align*}
		 Using the definition of the $W\left(\Omega,\,\rho ,\,\rho \Phi\right)$ norm, we derive
		$$|a(u,\,v)|\leq 4 \|u\|_{W(\Omega,\,\rho ,\,\rho \Phi)}\|v\|_{W(\Omega,\,\rho ,\,\rho \Phi)}.$$ Hence, $a$ is a continuous bilinear functional.\\
		Let us consider the linear operator $l$ from $W(\Omega,\,\rho \Phi)$ to $\mathbb{R}$ by $l(v)=\int_{\Omega}fv\rho dx_1dx_2$. The use of H\"older inequality and the definition of  the $W\left(\Omega,\,\rho ,\,\rho \Phi\right)$ norm yield  $|l(v)|\leq \|f\|_{L^2(\Omega,\,\rho )}\|v\|_{L^2(\Omega,\,\rho )}\leq C\|v\|_{W\left(\Omega,\,\rho ,\,\rho \Phi\right)},$ where $C=\|f\|_{L^2(\Omega,\,\rho )}<\infty,$ for $f\in L^2(\Omega,\,\rho )$. Therefore, $l$ is a continuous linear functional.  Thus, from the Lax-Milgram Theorem \cite[Theorem II.2.5 ]{BOYER}, there exists a unique $u\in W\left(\Omega,\,\rho ,\,\rho \Phi\right)$ such that $a(u,v)=l(v),$ for all $v\in W\left(\Omega,\, \rho ,\rho \Phi\right).$ 
	\end{proof}
	
	\begin{proposition}
		$L^2(\Omega,\,\rho )$ has an orthogonal basis formed by eigenfunctions of $L$. Moreover, its eigenvalues are positive and can be ordered in a sequence tending to $\infty$.
	\end{proposition}
	\begin{proof}
		Let $T=L^{-1}$ be the linear operator defined from $L^2(\Omega,\,\rho )$ to $W\left(\Omega,\,\rho ,\,\rho \Phi\right)$ by:\\ for $f\in L^2(\Omega,\,\rho )$, $T(f)=u$, where $u$ is the weak solution of (\ref{system}).\\For $f\in L^2(\Omega,\,\rho ),$
		\begin{align*}
			\|T(f)\|_{W\left(\Omega,\,\rho ,\,\rho \Phi\right)}^2&=\int_{\Omega}(\nabla u)^t\rho \Phi\nabla udx_1dx_2+\int_{\Omega}u^2\rho dx_1dx_2\\
			&\leq  \int_{\Omega}(\nabla u)^t\rho \Phi\nabla udx_1dx_2+\int_{\Omega}\left(2+{x_1^2+x_2^2}\right)u^2\rho dx_1dx_2\\
			&=\int_{\Omega}fu\rho dx_1dx_2.
		\end{align*}
		The use of the H\"older inequality yields 
		$\|T(f)\|_{W\left(\Omega,\,\rho ,\,\rho \Phi\right)}^2\leq \|f\|_{L^2(\Omega,\,\rho )}\|u\|_{L^2(\Omega,\,\rho )}$. Using the fact that $T(f)=u$, we obtain $\|T(f)\|_{W\left(\Omega,\,\rho ,\,\rho \Phi\right)}\leq \|f\|_{L^2(\Omega,\,\rho )}$. Therefore, the linear operator $T$ is continuous. Since $\rho $ satisfies the Pearson equation (\ref{te2a}) as well as the Neumann-type condition (\ref{bc}) we derive from the Theorem \ref{th1} that $W\left(\Omega,\,\rho ,\,\rho \Phi\right)$ is compactly embedded into $L^2(\Omega,\,\rho )$. Thus, $T$ is compact from $L^2(\Omega,\,\rho )$ to $L^2(\Omega,\,\rho ).$ Let us prove that $T$  is self-adjoint. Let  $f$ and $g$ be elements of  $L^2(\Omega,\,\rho )$, and $u=T(f),\,v=T(g)$. From (\ref{weakHelm})
		\begin{align*}
			\langle f,Tg\rangle_{L^2\left(\Omega,\,\rho \right)}&=\int_{\Omega}fv\rho dx_1dx_2\\
			&=\int_{\Omega}(\nabla v)^t\rho \Phi\nabla udx_1dx_2+\int_{\Omega}\left(2+{x_1^2+x_2^2}\right)vu\rho dx_1dx_2.
		\end{align*}
		Since $\Phi$ is symmetric, we obtain 
		$$\langle f,Tg\rangle_{L^2\left(\Omega,\,\rho \right)}=\int_{\Omega}(\nabla u)^t\rho \Phi\nabla vdx_1dx_2+\int_{\Omega}\left(2+{x_1^2+x_2^2}\right)vu\rho dx_1dx_2.$$
		Using the fact that $v$ is a weak solution of (\ref{system}) we deduce
		\begin{align*}
			\int_{\Omega}(\nabla u)^t\rho \Phi\nabla vdx_1dx_2+\int_{\Omega}\left(2+{x_1^2+x_2^2}\right)vu\rho dx_1dx_2
			&=
			\int_{\Omega}gu\rho dx_1dx_2\\
			&=\langle Tf,g\rangle_{L^2\left(\Omega,\,\rho \right)}.
		\end{align*} 
		Therefore, $\langle f,T(g)\rangle_{L^2\left(\Omega,\,\rho \right)}=\langle T(f),g\rangle_{L^2\left(\Omega,\,\rho \right)}$. That is $T$ is self-adjoint.\\
		So, $T$ is a compact self-adjoint operator from $L^2\left(\Omega,\,\rho \right)$ to $L^2\left(\Omega,\,\rho \right)$. Thus,  $L^2\left(\Omega,\,\rho \right)$ has an orthogonal basis $\{u_n\}_{n \geq 0}$ formed by eigenfunctions of $T$ (cf. \cite[Theorem II.6.5]{BOYER}). Moreover, the sequence $\{\mu_n\}_{n \geq 0}$ of eigenvalues of $T$, which are real numbers,   tends to 0.  
		Since $T=L^{-1}$, $Tu_n=\mu_n u_n\implies Lu_n=\nu_n u_n,\;\nu_n=\tfrac{1}{\mu_n}$.
		Therefore, $L^2\left(\Omega,\,\rho \right)$  has an orthogonal basis formed by eigenfunctions of $L$ and its eigenvalues  can be ordered in a sequence tending to $+\infty$. Moreover $\nu_n\|u_n\|^2_{L^2\left(\Omega,\,\rho \right)}=\langle Lu_n,u_n\rangle_{L^2\left(\Omega,\,\rho \right)}=a(u_n,\,u_n)$
		with $$a(u_n,\,u_n)=\int_{\Omega}(\nabla u_n)^t\rho \Phi\nabla u_ndx_1dx_2+\int_{\Omega}\left(2+{x_1^2+x_2^2}\right)u_n^2\rho dx_1dx_2.$$
		Since the matrix $\Phi$ is positive-definite, the right-hand side is positive. So, $\nu_n\geq 0.$
	\end{proof}
	
	\subsubsection*{Disclosure statement:The authors report there are no competing interests to declare}
	\subsubsection*{Funding details: No funding was obtained}

\end{document}